\documentclass[a4paper]{amsart}

\usepackage{amsmath}
\usepackage{amsfonts}
\usepackage{amsthm}
\usepackage{graphicx}
\usepackage{eucal}
\usepackage{amscd}
\usepackage[all,2cell]{xy}
\usepackage{amssymb}
\usepackage{mathrsfs}

 \usepackage{tikz}
 \usetikzlibrary{positioning}
 \tikzset{mynode/.style={draw,circle,inner sep=1pt,outer sep=0pt}}

\newtheorem{teo}{Theorem}[section]
\newtheorem{cor}[teo]{Corollary}
\newtheorem{lem}[teo]{Lemma}

\newtheorem{prop}[teo]{Proposition}
\newtheorem{example}[teo]{Example}
\newtheorem{remark}[teo]{Remark}

\newenvironment{Proof}{{\sc Proof.}\ }{\hfill$\square$\vspace{0.4truecm}}

\newcommand{\id}{\mbox{\rm id}}
\newcommand{\dom}{\mbox{\rm dom}\,}
\newcommand{\cod}{\mbox{\rm cod}\,}
\DeclareMathOperator{\op}{op}
\DeclareMathOperator{\Aut}{Aut}
\DeclareMathOperator{\End}{End}
\DeclareMathOperator{\ad}{ad}
\DeclareMathOperator{\Der}{Der}
\newcommand{\Gp}{\mathsf{Gp}}

\newcommand{\Set}{\mathsf{Set}}
\newcommand{\Cat}{\mathsf{Cat}}
\newcommand{\Ab}{\mathsf{Ab}}
\newcommand{\lMod}{\mbox{\rm -}{\mathsf{Mod}}}

\newcommand{\Mod}{\operatorname{Mod-\!}}

\newcommand{\Cal}[1]{{\mathcal #1}}

\newcommand{\DiGp}{\mathsf{DiGp}}
\newcommand{\SKB}{\mathsf{SKB}}

\newcommand{\im}{\mbox{\rm im}}

\newdir{ |>}{{}*!/-3.5pt/@{|}*!/-8pt/:(1,-.2)@^{>}*!/-8pt/:(1,+.2)@_{>}}

\begin{document}

\title{Semidirect products in Universal Algebra}

\author{Alberto Facchini}
\address[Alberto Facchini]{Dipartimento di Matematica ``Tullio Levi-Civita'', Universit\`a di Padova, Via Trieste 63, 35121 Padova, Italy}
\thanks{\thanks{The first author was partially supported by Ministero dell'Universit\`a e della Ricerca (Progetto di ricerca di rilevante interesse nazionale ``Categories, Algebras: Ring-Theoretical and Homological Approaches (CARTHA)''), and the Department of Mathematics ``Tullio Levi-Civita'' of the University of Padua (Research programme DOR1828909 ``Anelli e categorie di moduli'').}
 }
\email{facchini@math.unipd.it}

\author{David Stanovsk\'y}
\address[David Stanovsk\'y]{Department of Algebra, Faculty of Mathematics and Physics, Charles University, Sokolovsk\'a 83, 18675 Prague 8, Czech Republic}
\thanks{}
\email{stanovsk@karlin.mff.cuni.cz}

\subjclass[2010]{}

\date{}

\maketitle

\begin{abstract} First of all, we recall the well known notion of semidirect product both for classical algebraic structures (like groups and rings) and for more recent ones (digroups, left skew braces, heaps, trusses). Then we analyse the concept of semidirect product for an arbitrary algebra $A$ in a variety $\Cal V$ of type~$\Cal F$. An inner semidirect-product  decomposition $A=B \ltimes\omega$ of $A$ consists of a subalgebra $B$ of $A$ and a congruence $\omega$ on $A$  such that $B$ is a set of representatives of the congruence classes of $A$ modulo $\omega$. An outer semidirect product is the restriction to $B$ of a functor from a suitable category $\Cal C_B$ containing $B$, called the enveloping category of $B$, to the category $\Set_*$ of pointed sets.
 \end{abstract}

\section{Introduction}

The standard construction of semidirect product $K\rtimes Y$ of two groups $K$ and $Y$ has been extended in a number of directions in Category Theory. One possibility is that it can be viewed as the lax 2-colimit of the diagram $K\colon Y\to\Cat$ (see \cite {Gray, SE}). As a second point of view, it can be seen as a Grothendick construction \cite{MLM, Th}. In this paper we present a different point of view, a new one, from the point of view of Universal Algebra. We begin (Section~\ref{2}) by presenting several examples, both the classical ones (groups, Lie algebras, etc.) and some more recent ones (digroups, skew braces, heaps, trusses). This is in order to allow the reader to understand our subsequent definitions. Then we pass to consider {\em inner} semidirect products of algebras. Let $A$ be an algebra in a variety $\Cal V$ of type~$\Cal F$. Inner semidirect-product decompositions of $A$ consist of a subalgebra $B$ of $A$ and a congruence $\omega$ on $A$. One has a semidirect-product decomposition $A=B \ltimes\omega$ when the subalgebra $B$ is a set of representatives of the congruence classes of $A$ modulo $\omega$. There is a one-to-one correspondence between semidirect-product decompositions of $A$ and idempotent endomorphisms of $A$ (Corollary~\ref{11}). In the study of semidirect-product decompositions of $A$, one immediately sees that an important role is played by the category $\Set_*$ of pointed sets (Remark~\ref{ru}). Another important role is played by totally idempotent elements of $B$, that is, the elements $b$ of $B$ for which the singletons $\{b\}$ are subalgebras of $B$. For instance, in the variety of lattices all elements are totally idempotent, in the variety of semigroups the totally idempotent elements are the idempotent ones, in the varieties of monoids and groups the unique totally idempotent element is the identity. We then pass to the study of {\em outer} semidirect products. Given an algebra $B$ in a variety $\Cal V$ of type $\Cal F$, an outer semidirect product of $B$ assigns a pointed set $(A_b,a_b)$ to each element $b\in B$, and assigns to each $n$-ary function symbol $f\in\Cal F$ and each $n$-tuple $(b_1,\dots, b_n)\in B^n$ a morphism
$f_{(b_1,\dots, b_n)}\colon(A_{b_1},a_{b_1})\times\dots\times(A_{b_n},a_{b_n})\to (A_{f(b_1,\dots, b_n)},a_{f(b_1,\dots, b_n)})$ 
in $\Set_*$. On the disjoint union $A=\bigcup_{b\in B} (A_b\times\{b\})$, there are natural $n$-ary operations defined by $$f((a_{b_1},b_1),\dots, (a_{b_n},b_n))=(f_{(b_1,\dots, b_n)}(a_{b_1},\dots,a_{b_n}), f(b_1,\dots, b_n))$$ for each $n$-ary function symbol $f\in \Cal F$, so that $A$ becomes an algebra of type $\Cal F$, and we require this algebra $A$ to be in the variety $\Cal V$.
Every such semidirect product $B\to\Set_*$ can be extended to a functor $\Cal C_B\to\Set_*$ from a suitable {\em enveloping category} $\Cal C_B$ of $B$ to the category 
$\Set_*$ of pointed sets (see~\ref{functor}).

\section{Examples and basic notions}\label{2}

\subsection{Groups}\label{2.1} The following is well known to all of us, but we repeat it here to fix the basic notions and notations. For a group $G$, a normal subgroup $K$ of $G$ and a subgroup $Y\le G$, the fact that $G$ is the {\em (inner) semidirect product} of $K$ and $Y$ can be equivalently stated saying that:

\begin{enumerate}
    \item[(a)] $G=KY$ and $K\cap Y=\{1_G\}$.
    \item[(b)] For every $g\in G$, there is a unique pair $(k,y)\in K\times Y$ such that $g = ky$.
    \item[(c)] For every $g\in G$, there is a unique pair $(k,y)\in K\times Y$ such that $g = yk$.
    \item[(d)] There is an idempotent endomorphism of $G$ with kernel $K$ and image $Y$.
    \item[(e)]   There is a homomorphism of $G$ onto $Y$ that is the identity on $Y$ and has kernel $K$.
    \item[(f)]   The mapping $Y\to G/K$, $y\mapsto yK$, is an isomorphism. 
\end{enumerate}
     
     Given any such semidirect-product decomposition $G=KY$ with $K\cap Y=\{1_G\}$, there is an {\em action} of $Y$ on $K$, that is, a group homomorphism $\phi\colon Y\to\Aut(K)$, defined by $\phi(y)(k)=yky^{-1}$ for every $y\in Y$, $k\in K$. 

        \medskip
        
Conversely, given any two groups $K$ and $Y$ and any group homomorphism $\phi\colon Y\to\Aut(K)$, it is possible to define a group structure on the cartesian product $K\times Y:=\{\,(k,y)\mid k\in K,\ y\in Y\,\}$ defining the operation by $$(k,y)(k',y')=(k\phi_y(k'), yy').$$ This group is the {\em outer semidirect product} of $K$ and $Y$, denoted by $K\rtimes_{\phi} Y$.

  \medskip

Any inner semidirect product $G=KY$ is isomorphic to the outer semidirect product $K\rtimes_{\phi} Y$ with $\phi$ the action of $Y$ on $K$ defined by conjugation: $$\phi(y)(k)=yky^{-1}$$ for every $y\in Y$, $k\in K$. 

\medskip

For a fixed group $G$, there is a bijection between the set $E$ of all idempotent endomorphisms of $G$ 
$$E:=\{\, e\mid e\colon G\to G\ \mbox{an endomorphism of $G$, and } e^2=e\,\}$$ and the set $A$ of all pairs $(K,Y)$, where $K$ is a normal subgroup of $G$, $Y$ is a subgroup of $G$, and $G$ is the inner semidirect product $ K\rtimes Y$: $$A:=\{\,(K,Y)\mid K\trianglelefteq G,\ Y\le G, \ G=KY, \ K\cap Y=\{1_G\}\,\}.$$  In this bijection, an idempotent endomorphism $e$ of $G$ corresponds to the pair $(\ker(e), \im(e))$.  

\medskip

For an additive {\em abelian} group $G$, idempotent endomorphisms $e\colon G\to G$ of $G$ correspond bijectively to the pairs $(K,Y)$, where $K=\ker(e)$, $Y=\im(e)$ and $G=K\oplus Y$ (direct-sum decomposition). This also holds for vector spaces and right modules $G$ over a ring: there is a one-to-one correspondence between the set of all idempotent endomorphisms of $G$ and the set of all pairs $(K,Y)$ of substructures of $G$ for which $G=K\oplus Y$.

\subsection{Lie algebras}\label{2.2}

For Lie algebras, semidirect products are sometimes called ``inessential extensions'' \cite[Chapitre Premier, \S 1, N.~7, D\'efinition 6 and N.~8]{Bou}. Inner semidirect-product decompositions of a Lie algebra $L$ over a commutative ring $k$ with identity are exactly the direct-sum decomposition $L=K\oplus Y$ of $L$ as a $k$-module, where $K$ is an ideal of $L$ and $Y$ is a Lie subalgebra of $L$. Such decompositions of $L$ bijectively correspond to idempotent endomorphisms of $L$. 

Given any such semidirect-product decomposition $L=K\oplus Y$ of $L$, the adjunction $\ad_L\colon L\to \Der(L)$, defined by $\ad_L(x)(y)=[x,y]$, induces by restriction  an {\em action} of $Y$ on $K$, that is, a Lie algebra homomorphism $\phi\colon Y\to\Der(K)$, defined by $\phi(y)(k)=[y,k]$ for every $y\in Y$, $k\in K$, into the Lie algebra $\Der(K)$ of derivations of $K$.  

Conversely, given any two Lie $k$-algebras $K$ and $Y$ and any Lie algebra morphism $\phi\colon Y\to\Der(K)$, it is possible to define a Lie algebra structure on the $k$-module $K\oplus Y$ (direct sum of the two $k$-modules $K$ and $L$),  defining the Lie bracket as $[(k,y),(k',y')]=([k,k']+\phi_y(k')-\phi_{y'}(k), [y,y'])$ \cite[Chapitre premier, \S 1, N.~8]{Bou}. This Lie algebra is  the outer semidirect product of $K$ and $Y$.

\subsection{Rings}\label{2.3}

For an associative ring $R$ (not-necessarily with an identity), idempotent endomorphisms of $R$ are in a one-to-one correspondence with the pairs $(K,S)$, where $K$ is an ideal of $R$, $S$ is a subring of $R$, and $R= K\oplus S$ as an abelian additive group. Now the natural $R$-$R$-bimodule structure $_RR_R$ on the ring $R$ induces an $S$-$K$-bimodule structure $_SK_K$ and a $K$-$S$-bimodule structure $_KK_S$ on $K$. Correspondingly, we have a ring homomorphism $\lambda\colon S\to\End_{\Mod K}(K_K)$  and a ring antihomomorphism $\rho\colon S\to\End_{K\lMod}(_KK)$ such that $\lambda(s)\circ\rho(t)=\rho(t)\circ\lambda(s)$ and $\rho(s)(x)y=x\lambda(s)(y)$  for all $s,t\in S$ and $x,y\in K$ (these conditions immediately follow from the associativity of multiplication in $R$).

Conversely, given two associative rings $K$ and $S$, a ring homomorphism $$\lambda\colon S\to\End_{\Mod K}(K_K)$$ and a ring antihomomorphism $$\rho\colon S\to\End_{K\lMod}(_KK)$$ such that $\lambda(s)\circ\rho(t)=\rho(t)\circ\lambda(s)$ and $\rho(s)(x)y=x\lambda(s)(y)$  for every $s,t\in S$ and every $x,y\in K$, it is possible to define a multiplication on the abelian group $K\oplus S$ setting $(k,s)(k',s')=(kk'+\lambda(s)(k')+\rho(s')(k), ss')$. 

\medskip

Let us check associativity of multiplication in $K\oplus S$. We have
\[ \begin{array}{l}((k,s)(k',s'))(k'',s'')=(kk'+\lambda(s)(k')+\rho(s')(k), ss')(k'',s'')= \\ \quad =((kk')k''+\lambda(s)(k')k''+\rho(s')(k)k''+\lambda(ss')(k'')+ \\ \quad\qquad\qquad +\rho(s'')(kk'+\lambda(s)(k')+\rho(s')(k)), (ss')s'')\end{array}\] and
\[ \begin{array}{l}(k,s)((k',s')(k'',s''))=(k,s)(k'k''+\lambda(s')(k'')+\rho(s'')(k'), s's'')= \\ \quad =(k(k'k''+\lambda(s')(k'')+\rho(s'')(k'))+\lambda(s)(k'k''+\lambda(s')(k'')+ \\ \quad\qquad\qquad +\rho(s'')(k'))+\rho(s's'')(k),s(s's'').\end{array}\]
Now $$(kk')k''=k(k'k'')$$ because $K$ is an associative ring,
$$\lambda(s)(k')k''=\lambda(s)(k'k'')$$ because $\lambda\colon S\to\End_{\Mod K}(K_K)$,
$$\rho(s')(k)k''=k\lambda(s')(k''),$$
$$\lambda(ss')(k'')=\lambda(s)\lambda(s')(k'')$$ because $\lambda$ is a ring morphism,
$$\rho(s'')(kk')=k\rho(s'')(k')$$ because $\rho\colon S\to\End_{K\lMod}(_KK)$,
$$\rho(s'')\lambda(s)(k')=\lambda(s)\rho(s'')(k')$$ because $\lambda(s)\circ\rho(s'')=\rho(s'')\circ\lambda(s)$,
$$\rho(s'')\rho(s')(k)=\rho(s's'')(k)$$ because $\rho$ is a ring antihomomorphism, and
$$(ss')s''=s(s's'')$$ because $S$ is an associative ring.

\medskip
  
Notice that now we need {\em two} ``actions'' $\lambda$ and $\rho$, a left one and a right one, while for groups and Lie algebras one action was sufficient. This is due to the fact that groups and Lie algebras are isomorphic to their opposite (the antiisomorphisms are $g\mapsto g^{-1}$ for groups and $x\mapsto -x$ for Lie algebras). This does not occur for associative algebras.

  \medskip

  What we have seen in this subsection about associative rings can be easily extended to the case of associative $C$-algebras, where $C$ is any commutative ring with identity.
  
  \subsection{Not-necessarily associative algebras}
  
Let $C$ be a commutative ring with identity $1$. With the term {\em $C$-algebra} we now mean a $C$-module $A$ with a further $C$-bilinear operation $$\cdot\colon A\times A\to A,\qquad (r,r')\mapsto r\cdot r'=rr'$$ (see \cite{Bou} and \cite[Section~2]{meltem}). Thus $C$-algebras generalize at the same time the Lie algebras we have considered in \ref{2.2} and the associative rings considered in \ref{2.3}. For a $C$-algebra  $A$, idempotent endomorphisms of $A$ are in a one-to-one correspondence with the pairs $(K,B)$, where $K$ is an ideal of $A$, $B$ is a sub-$C$-algebra of $A$, and $A= K\oplus B$ as a $C$-module. Then we have two $C$-module morphisms $\lambda,\rho\colon B\to\End_{\Mod C}(K)$.

Conversely, given two $C$-algebras $K$ and $S$, and two $C$-module morphisms $$\lambda,\rho\colon B\to\End_{\Mod C}(K),$$ it is possible to define a multiplication on the $C$-module $K\oplus B$ setting $(k,b)(k',b')=(kk'+\lambda(b)(k')+\rho(b')(k), bb')$ for every $k,k'\in K$ and every $b,b'\in B$. 

\medskip

For applications of idempotent endomorphisms in the study of algebras with a bilinear form, see~\cite{Leila}.

 \subsection{Digroups}\label{dig}

The content of this subsection is taken from \cite{FacSkew}. A {\sl digroup} is
a triple $(A, *,\circ)$, where $(A, *) $ and $(A,\circ)$ are groups (not-necessarily abelian) for which $1_{(A,*)}=1_{(A,\circ)}$. An {\em ideal} in a digroup $(A, *,\circ)$ is a subset $I$ of $A$ such that $I$ is a normal subgroup of $(A, *)$, $I$  is a normal subgroup of $(A, \circ)$, and  $a*I=a\circ I$ for every $a\in A$  (the cosets modulo $I$ with respect to the two operations coincide).
For a digroup $(A, *,\circ)$, there is a one-to-one correspondence between the set of all ideals of $A$ and the set of all the equivalence relations on $A$ compatible with both the operations $*$ and $\circ$. It associates with any equivalence relation $\sim$ on the set $A$ compatible with both the operations $*$ and $\circ$  of $A$ the equivalence class $[1_A]_\sim$ of the identity of $A$, which is an ideal of the digroup $A$. Conversely, it associates with any ideal $I$ of $A$, the relation $\sim_I$ on $A$ defined, for every $a,b\in A$, by $a\sim_I b$ if $a*b^{-*}\in I$. Here $a^{-*}$ denotes the inverse of $a$ in the group $(D,*)$, and $a^{-\circ}$ will denote the inverse of $a$ in $(D,\circ)$.

Since $1_{(A,*)}=1_{(A,\circ)}$, the kernel of a morphism is defined as the inverse image of $1_{(A,*)}=1_{(A,\circ)}$. The kernel of any morphism is an ideal, as is easily seen.

		There is a forgetful functor $U$ of the category $\DiGp$ of digroups into the category $\Set_*$ of pointed sets, which assigns to each digroup $D$ the pointed set $(D,1)$. Notice that if we have any two group structures $(D,*,1_*)$ and $(D,\circ,1_\circ)$ on the same set $D$, and we take any isomorphism $\varphi\colon (D,1_*)\to(D,1_\circ)$ in $\Set_*$, we can lift the group structure of $(D,\circ,1_\circ)$ to the pointed set $(D,1_*)$, obtaining a digroup structure $(D,*,\circ,1_*)$ on the set $D$.

\medskip

	Given a digroup $(D,*,\circ)$ and an element $a$ of $D$, let $\lambda^D_a\colon D\to D$ denote the mapping defined by
$$\lambda^D_a(b)=a^{-*}*(a\circ b)$$ for every $b\in D$. Then every $\lambda^D_a$ is an automorphism of $U(D)$ in the category $\Set_*$. Its inverse is the isomorphism $(D,1)\to (D,1)$, $c\mapsto a^{-\circ}\circ(a*c)$.  (This is the mapping $\lambda^D_a$ relative to the digroup $(D,\circ,*)$ with the operation $*$ and $\circ$ swapped. Notice that the digroups  $(D,*,\circ)$ and $(D,\circ,*)$ are not isomorphic in the category $\DiGp$.)

The mapping $\lambda^D\colon (D,1)\to (\Aut_{\Set_*}(D,1),\id_D)$, $\lambda^D\colon a\in D\mapsto\lambda_a^D$, is a morphism in the category $\Set_*$ of the pointed set $(D,1)$ into the automorphism group $ \Aut_{\Set_*}(D,1)$ of all automorphisms of  the pointed set $(D,1)$. The property ``$a*I=a\circ I$ for every $a\in D$'' in the definition of ideal $I$ of a digroup $D$ is equivalent to ``$\lambda_a^D(I)\subseteq I$  for every $a\in D$''.
 
 \medskip
 
Any intersection of ideals is an ideal, so that we get a complete lattice $\Cal I(A)$ of ideals for any digroup $A$.  In particular, every subset $X$ of $A$ generates an ideal, the intersection of the ideals that contain $X$.  

\medskip

	Fix a digroup $(D,*,\circ)$ and an idempotent endomorphism $e$ of $D$. Let $B$ be the image of $e$, so that $B$ is a  subdigroup of $D$, and let $I$ be the kernel of $e$, which is an ideal of $D$. Since $e$ is an idempotent group endomorphism with respect to both group structures on $D$, we get that $D$ is a semidirect product of $B$ and $I$ with respect to both groups structures on $D$, so $D=I*B=\{\, i*b\mid i\in I,b\in B\,\}$, $D=I\circ B=\{\, i\circ b\mid i\in I,b\in B\,\}$ and $I\cap B=\{1_D\}$. 
	
\begin{prop}\label{0.1} {\rm \cite[Proposition 2.2]{FacSkew}} Let $(D,*,\circ)$ be  a digroup, $B$ a subdigroup of $D$ and $I$ an ideal of $D$. The following conditions are equivalent:

	\begin{enumerate}
		\item $D=B\circ I$ and $B\cap I = \{1_D\}$.
		\item For every $a\in D$, there is a unique pair $(b,i_1)\in B\times I $ such that $a=b\circ i_1$.
		\item For every $a\in D$, there is a unique pair $(b,i_2)\in B\times I $ such that $a=i_2\circ b$.
		\item $D=B*I$ and $B\cap I = \{1_D\}$.
		\item For every $a\in D$, there is a unique pair $(b,i_3)\in B\times I $ such that $a=b*i_3$.
		\item For every $a\in D$, there is a unique pair $(b,i_4)\in B\times I $ such that $a=i_4* b$.
			\item There is an idempotent digroup  endomorphism $e$ of $D$ whose image is $B$ and whose kernel is $I$.
	\end{enumerate}
 \end{prop}
The digroup $D$ is the {\em inner semidirect product} of its subdigroup $B$ and its ideal $I$ if any of the equivalent conditions of the previous proposition holds.

\medskip

By Proposition \ref{0.1}, every element $a\in D$ can be written in a unique way in the four forms $a=b\circ i_1=i_2\circ b=b* i_3=i_4* b$, where $b=e(a)$ and $i_1,i_2,i_3,i_4\in I$. Now it is easy to prove \cite{FacSkew} that the elements $i_2,i_3,i_4$ depend only on $i_1$, via the formulas \begin{equation} i_2=(\phi_{\circ b})^{-1}(i_1), \qquad i_3=\lambda_b(i_1),\qquad{\mbox{\rm and}} \qquad i_4=(\phi_{* b})^{-1}(\lambda_b(i_1)).\label{sistema}\end{equation} 	
Here and in the rest of this subsection we will write conjugation on the right, as one often does in Group Theory, so that, for any two groups $(Y, \otimes)$ and $(K,\otimes)$ and any group antihomomorphism $\phi_\otimes\colon (Y, \otimes)\to\Aut_\Gp(K,\otimes)$, we denote by $(Y\ltimes_{\phi_\otimes}K, \otimes)$ the semidirect product of $Y$ and $K$ via $\phi_\otimes$, where $ \otimes$ is defined by $$(y,k) \otimes (y',k'):=(y\otimes y', \phi_{\otimes y'}(k)\otimes k').$$

Suppose we have a digroup $(D,*,\circ)$. We want to describe the semidirect-product decompositions of $D$. Equivalently, we want to study the idempotent digroup endomorphisms $e$ of $D$. Such idempotent endomorphisms $e$ have a kernel $K$ and an image $Y$, and $D$ turns out to be a semidirect product $D=Y\ltimes K$. By the formulas (\ref{sistema}),  three mappings immediately appear in a natural way:

(1) The group antihomomorphism $\phi_*\colon (Y, *)\to\Aut_\Gp(K,*)$ related to the semi\-direct-product decomposition $D=Y\ltimes K$ of the group $(D,*)$. It is defined by $\phi_*\colon y\mapsto\phi_{*y}$, where $\phi_{*y}(k)=y^{-*}*k*y$ for every $y\in Y$ and every $k\in K$, and is induced by the conjugation in the group $(D,*)$.

(2) The group antihomomorphism $\phi_\circ\colon (Y, \circ)\to\Aut_\Gp(K,\circ)$ related to the semi\-direct-product decomposition $D=Y\ltimes K$ of the group $(D,\circ)$. It is defined by $\phi_\circ\colon y\mapsto\phi_{\circ y}$, where $\phi_{\circ y}(k)=y^{-\circ}\circ k\circ y$ for every $y\in Y$ and $k\in K$, and is induced by the conjugation in the group $(D,\circ)$.

(3) The morphism $\Lambda\colon (Y,1)\to (\Aut_{\Set_*}(K,1),\id_K)$ in the category $\Set_*$, defined by $\Lambda\colon y\mapsto\Lambda_y$, where $\Lambda_y(k)=y^{-*}*(y\circ k)$ for every $y\in Y$, $k\in K$. It is induced by the morphism $\lambda\colon  (D,1)\to (\Aut_{\Set_*}(D,1),\id_D)$ ($\lambda$ is a morphism in $\Set_*$), relative to the digroup $(D,*,\circ)$. Here $\Aut_{\Set_*}(X,x_0)$ denotes the group of all automorphisms of the pointed set $(X,x_0)$. Each $\Lambda_y$ is a permutation of $K$ because its two-sided inverse is the mapping $K\to K$, $k\mapsto y^{-\circ}\circ(y*k)$. Notice that $\Lambda_1\colon K\to K$ is the identity mapping of $K$ and $\Lambda_y(1)=1$ for every $y\in Y$.

\begin{teo}\label{the truth} {\rm \cite[Theorem 3.1]{FacSkew}} Let $(D,*,\circ)$ be any digroup, and suppose that $D$ is the semidirect product of its subdigroup $Y$ and its ideal $K$ (i.e., that the equivalent conditions of Proposition~{\rm \ref{0.1}} are satisfied). Let $\phi_*$, $\phi_\circ$ and $\Lambda$ be the three mappings defined above. Then there is a digroup isomorphism $$\alpha\colon (Y\times K,+, \circ)\to (D,*,\circ)$$ defined by $\alpha(y,k)= y\circ k$ for every $y\in Y$ and  every $k\in K$, where the operations on the digroup $(Y\times K,+,\circ)$ are defined by \begin{equation}(y,k)+(y',k')=(y* y', (\Lambda_{y* y'})^{-1}(\phi_{* y'}(\Lambda_y(k))* \Lambda_{y'}(k')))
    \label{A}
\end{equation}  and \begin{equation}(y,k) \circ(y',k')=(y\circ y', \phi_{\circ y'}(k)\circ k')\label{B}
\end{equation} for every $(y,k),(y',k')\in Y\times K$. \end{teo}

 The operations $+$ and $ \circ$ in the statement of Theorem~\ref{the truth} may seem quite complicate, but their naturality appears as soon as we look at how they define the products between an element of the form $(y,1)$ and an element of the form $(k,1)$. In fact, it is very easy to see, from the two formulas in Theorem~\ref{the truth}, that for every $y\in Y$ and every $k\in K$: \begin{equation}\begin{array}{l} (y,1) \circ (1,k)=(y,k) \\ (1,k) \circ(y,1)=(y,\phi_{\circ y}(k)) \\
(y,1)+(1,k)=(y,(\Lambda_y)^{-1}(k)) \\ (1,k)+(y,1)=(y,(\Lambda_y)^{-1}\phi_{*y}(k)).\end{array}\label{ppp}\end{equation}

 Of these four identities, the first is a trivial identity, the second allows to define the antihomomorphism $\phi_\circ$, the third allows to define the morphism $\Lambda$, and the fourth, in view of the third, allows to define the antihomomorphism $\phi_*$. Clearly, the last three identities in (\ref{ppp}) correspond to the identities  in (\ref{sistema}). 
 
Finally, here is the outer semidirect product of two digroups:

\begin{teo}\label{thm2} {\rm \cite[Theorem 4.1]{FacSkew}} Let $(Y,*,\circ)$ and $(K,*,\circ)$ be digroups, $\phi_*\colon (Y, *)\to\Aut_\Gp(K,*)$ and $\phi_\circ\colon (Y, \circ)\to\Aut_\Gp(K,\circ)$ be two group antihomomorphisms, and $\Lambda\colon (Y,1)\to (S_K,\id)$ be a morphism in $\Set_*$. On the cartesian product $Y\times K$ define two operations via \begin{equation}(y,k)+(y',k')=(y* y', (\Lambda_{y* y'})^{-1}(\phi_{* y'}(\Lambda_y(k))* \Lambda_{y'}(k')))\label{+}\end{equation} 
and \begin{equation}(y,k) \circ(y',k')=(y\circ y', \phi_{\circ y'}(k)\circ k')\label{circ}\end{equation}  for every $y,y'\in Y$ and $k,k'\in K$. Then $(Y\times K,+,\circ)$ is a digroup.\end{teo}

\begin{cor}\label{2.5} Let $(Y,*,\circ)$ and $(K,*,\circ)$ be digroups, $\phi_*\colon (Y, *)\to\Aut_\Gp(K,*)$ and $\phi_\circ\colon (Y, \circ)\to\Aut_\Gp(K,\circ)$ be two group antihomomorphisms,\linebreak $\Lambda\colon (Y,1)\to (S_K,\id)$ be a morphism in $\Set_*$, and construct the correponding semidirect product $(Y\times K,+,\circ)$. Then $(Y\times K,+,\circ)$ is canonically the direct product of the digroups $Y$ and $K$ if and only if the three mappings $\phi_*$, $\phi_\circ$ and $\Lambda$ are all the constant mapping that sends every element of $Y$ to the identity $\id_K$
    of $K$.
\end{cor} 

\begin{Proof} If $(Y\times K,+,\circ)$ is the direct product, then \begin{equation}k\circ k'= \phi_{\circ y'}(k)\circ k'\quad\mbox{\rm and}\quad k*k'=(\Lambda_{y* y'})^{-1}(\phi_{* y'}(\Lambda_y(k))* \Lambda_{y'}(k'))
\label{circ1}\end{equation} for every $y,y'\in Y$ and $k,k'\in K$. Now the first equation in~(\ref{circ1}) is clearly equivalent to $k=\phi_{\circ y'}(k)$ for every $k$, i.e.~that $\phi_\circ$ is the constant mapping that sends every element of $Y$ to $\id_K$. The second equation in~(\ref{circ1}) can be rewritten as $$\phi_{* y'}(\Lambda_y(k))* \Lambda_{y'}(k')=\Lambda_{y* y'}(k*k'),$$ and from this equation we get the equalities \begin{equation}\phi_{* y'}(\Lambda_y(k))=\Lambda_{y* y'}(k)\ \mbox{\rm (for $k'=1$)}\quad \mbox{\rm and}\quad \Lambda_{y'}(k')=\Lambda_{y* y'}(k')\ \mbox{\rm (for $k=1$).}
\label{xxx}\end{equation}  From the last equation in~(\ref{xxx}), we get for $y=(y')^{-*}$ that $$\Lambda_{y'}(k')=\Lambda_{(y')^{-*}* y'}(k')=k',$$ i.e., that $\Lambda$  is the constant mapping that sends every element of $Y$ to $\id_K$. From this and the first equation in~(\ref{xxx})
we get that $\phi_{* y'}$ is also the identity automorphism of $K$.

The inverse implication is trivial.
\end{Proof}
 \subsection{Left skew braces}\label{5.1}
 
A {\sl (left) skew brace} is
a triple $(A, *,\circ)$, where $(A, *) $ and $(A,\circ)$ are groups (not necessarily abelian) and \begin{equation}a\circ (b * c) = (a\circ b)*a^{-*}* ( a\circ c)\label{lsb}\end{equation}
for every $a,b,c\in A$.  It is not difficult to show that left skew braces are digroups, so that all what we have seen for digroups applies to left skew braces.

\medskip

%

%

%

Clearly, left skew braces form an algebraic variety, whose homomorphisms $A\to A'$ are the mappings $f\colon A\to A'$ such that $f(a*b)=f(a)* f(b)$ and $f(a\circ b)=f(a)\circ f(b)$ for every $a,b\in A$.  In particular, we have the category $\SKB$ of all left skew braces.

\begin{lem}\label{Bachi}{\rm \cite{Bac}} A digroup $(A,*,\circ)$ is a left skew brace if and only if the mapping $\lambda\colon (A,\circ)\to\Aut (A,*)$, given by $\lambda\colon a\mapsto\lambda_a$, where $\lambda_a(b)=a^{-*}* (a\circ b)$, is a group morphism.\end{lem}
 
We say that the left skew brace $D$ is the {\em inner semidirect product} of its subbrace $B$ and its ideal $I$ if any of the equivalent conditions of Proposition~\ref{0.1} holds.

\medskip

There is a bijection between the set of all idempotent left skew brace endomorphisms of a left skew brace $D$ and the set of all pairs $(B,I)$, where $B$ is a  subleft skew brace of $D$, $I$ is an ideal of $D$ and $D$ is the semidirect product of $B$ and $I$. This bijection associates with each idempotent endomorphism $e$ the pair $(e(D),\ker(e))$.
	
	\medskip

	If $(D,*,\circ)$ is a left skew brace, $e$ is an idempotent group endomorphism of either $(D,*)$ or $(D,\circ)$, the kernel of $e$ is an ideal of the left skew brace $(D,*,\circ)$ and the image of $e$ is a subleft skew brace, then $e$ is an idempotent left skew brace endomorphism of $(D,*,\circ)$.

\begin{prop} {\rm \cite[Proposition 4.2]{FacSkew}} Let $(Y,*,\circ)$ and $(K,*,\circ)$ be left skew braces and suppose that there are  two group antihomomorphisms $\phi_*\colon (Y, *)\to\Aut_\Gp(K,*)$ and $\phi_\circ\colon (Y, \circ)\to\Aut_\Gp(K,\circ)$, and a group homomorphism $\Lambda\colon (Y,\circ)\to\Aut_\Gp(K,*)$. Then their semidirect product (as digroups) is a left skew brace if and only if $$\begin{array}{l}
        \phi_{\circ (y'*y'')}(k)\circ (\Lambda_{y'*y''})^{-1}\bigl(\phi_{*y''}(\Lambda_{y'}(k'))*\Lambda_{y''}(k'')\bigr) = \\ \quad=
        (\Lambda_{y\circ (y'*y'')})^{-1}\bigl(\phi_{*\lambda_{y}^Y(y'')}(\Lambda_{y\circ y'}(\phi_{\circ y'}(k)\circ k')* \\ \quad*\Lambda_{y}(k)^{-*})\bigr)*\Lambda_{y\circ y''}(\phi_{\circ y''}(k)\circ k'').
    \end{array}$$
\end{prop}

Consider the inclusion functor $S\colon\SKB\to \DiGp$ from the category of left skew braces to the category of digroups. This functor $S$ has a left adjoint, which assigns with every digroup $D$ the quotient digroup $D/I$, where $I$ is the ideal of the digroup $D$ generated by the subset $\{\,  (a\circ b)*a^{-1}* ( a\circ c)*(a\circ (b * c))^{-*}\mid a,b,c\in D\,\}$. Clearly, $D/I$ is a left skew brace, and the category $\SKB$ of left skew braces is a reflective subcategory of the category $\DiGp$ of digroups.

\medskip

The commutator of two ideals $I$ and $J$ of a left skew brace $A$ is the ideal of $A$ generated by the commutator $[I,J]_{(A,*)}$ in the group $(A,*)$, the commutator $[I,J]_{(A,\circ)}$ in the group $(A,\circ)$, and all the elements $(i\circ j)^{-*}*i*j$ with $i\in I$ and $j\in J$ \cite{BFP}. Notice the clear relation between these three sets whose union generates the commutator and the three conditions necessary and sufficient for a semidirect product of digroups (left skew braces) to be a direct product (Corollary~\ref{2.5}). 

\begin{lem}\label{comm} For every pair of ideals $I$ and $J$ of a left skew brace $A$, one has $[I,J]=[J,I]$.\end{lem}

\begin{proof} For groups, one has that $[I,J]_{(A,*)}=[J,I]_{(A,*)}$ and $[I,J]_{(A,\circ)}=[J,I]_{(A,\circ)}$. Hence it remains to show that if $(A,*,\circ)$ is a left skew brace with the groups $(A,*)$ and $(A,\circ)$ abelian and $(i\circ j)^{-*}*i*j=1$ for every $i\in I$ and $j\in J$, then  $(j\circ i)^{-*}*j*i=1$ for every $i\in I$ and $j\in J$. Now $(i\circ j)^{-*}*i*j=1$ implies that $i\circ j=i*j$, so that $j\circ i=i\circ j=i*j=j*i$. Hence $(j\circ i)^{-*}*j*i=1$.\end{proof}

Recall that in the lattice of all the ideals of a left skew brace $A$, one has that $I\vee J=I*J=I\circ J$ and $I\wedge J=I\cap J$. This is a multiplicative lattice \cite{FFJ} in which multiplication is commutative in view of Lemma~\ref{comm}.

\begin{lem} If $I,J,K$ are ideals of a left skew brace $A$, then $[I,J*K]=[I,J]*[I,K]$.\end{lem}

\begin{proof} From $J*K\supseteq J,K$, it follows that $[I,J*K]\supseteq [I,J],[I,K]$. Therefore $[I,J*K]\supseteq [I,J]*[I,K]= [I,J]\circ [I,K]$. Now it is known that the equality $[I,JK]=[I,J][I,K]$ holds for normal subgroups $I,J,K$ of a group $G$. Hence, to conclude the proof, it suffices to prove that if $(A,*,\circ)$ is a brace in which both the groups $(A,*)$ and $(A,\circ)$ are abelian, and $i*j=i\circ j$, $i*k=i\circ k$ for every $i\in I$, $j\in J$ and $k\in K$, then $i*(j\circ k)=i\circ(j\circ k)$ for every $i,j,k$. Now $i*j=i\circ j$ for every $i$ and $j$ can be restated saying that $J$ is contained in the kernel of the group morphism $\lambda|^I\colon (A,\circ)\to\Aut(I,*)$. Similarly, $K$ is contained in the kernel of the group morphism $\lambda|^I$. Therefore $J\circ K$ is contained in that kernel, that is $i*x=i\circ x$ for every $x\in J\circ K=J*K$, as desired.
\end{proof}

The center of a left skew brace $A$ is the greatest ideal $Z$ of $A$ such that $[Z,A]=1$. Hence $Z=\{\, z\in A\mid a*z=z*a,\ a\circ z=z\circ a$ and $a*z=a\circ z$ for every $a\in A\,\}$. Again, notice the relation between these three conditions for elements in the center of a left skew brace and the three conditions necessary and sufficient for a semidirect product of digroups (in particular, of left skew braces) to be a direct product (Corollary~\ref{2.5}). 

\subsection{Heaps} 

The content of this subsection is taken from \cite{FacHeaps}, to which we refer for further details. A {\em heap} is a set $X$ endowed with a ternary operation $[-,-,-]\colon X\times X\times X\to X$ that is  a {\em Mal'tsev operation}, that is, $[x,x,y]=y$ and $[x,y,y]=x$ for every $x,y\in X$, and associative (i.e., 
$[[x,y,z],w,u]=[x,y,[z,w,u]]$ for every $x,y,z,w,u\in X$).

 \smallskip
 
 As usual in Universal Algebra, a mapping $f\colon (X,[-,-,-])\to (X',[-,-,-])$ between two heaps is a {\em heap homomorphism} if $$f([x,x',x''])=[f(x), f(x'), f(x'')]$$ for every $x,x',x''\in X$, and a subset $S$ of a heap is a {\em subheap} if $[x, y, z] \in S$ for every $x, y, z\in S$.

For example, fix any group $G$ and define a ternary operation $[-,-,-]$ on $G$ setting $[x,y,z]=xy^{-1}z$ for every $x,y,z\in G$.  Then $(G,[-,-,-])$ is a heap. Moreover, every non-empty heap is of this form, and there is a natural functor $(G,\cdot)\mapsto(G,[-,-,-])$ of the category of groups into the categories of heaps.  Nevertheless these two categories are not equivalent.

 \smallskip

 A subheap $S$ of a heap $X$ is a {\em normal subheap} if it is non-empty and $[[x, e, s], x, e]\in S$ for every $x\in X$ and every $e,s\in S$.

There is an onto mapping $\{\,$normal subheaps$\,\}\to \{\,$congruences$\,\}$, $S\mapsto\,\sim_S$, which  is not a bijection in general. 

More precisely, let $X$ be a heap. On the set ${\mathcal{ N}}(X)$ of all normal subheaps of $X$ define a pre-order $\preceq$ setting, for all $M,N\in {\mathcal{ N}}(X)$, $M\preceq N$ if for every $x,y\in X$ and $s\in M$ such that $[x, y,s] \in M$ there exists $t\in N$ such that $[x, y,t] \in N$. Let $\simeq$ be the equivalence relation on 
     ${\mathcal{ N}}(X)$ associated to the pre-order $\preceq$. Then the partially ordered set ${\mathcal{ N}}(X)/\!\simeq$ is order isomorphic to the partially ordered set ${\mathcal{ C}}(X)$ of all congruences of the heap $X$. 

\begin{prop}\label{boh} Let $X\ne\emptyset$ be a heap, $Y$ be a subheap of $X$, and $\omega$ a congruence on $X$. The following conditions are equivalent:

{\rm (a)} $Y$ is a set of representatives of the equivalence classes of $X$ modulo $\omega$, that is, $Y\cap[x]_\omega$ is a singleton for every $x\in X$.

{\rm (b)} There exists an idempotent heap endomorphism of $X$ whose image is $Y$ and whose kernel is $\omega$.

{\rm (c)} For every $a\in X$ and every $c\in Y$ there exist a unique element $b\in X$ and a unique element $d\in Y$ such that $a=[b,c,d]$ and $b\,\omega\, c$.

  {\rm (d)}
For every $a\in X$ and every $c\in Y$ there exist a unique element $b\in X$ and a unique element $d\in Y$ such that $a=[d,c,b]$ and $b\,\omega\, c$.

{\rm (e)} The mapping $g\colon Y\to X/\omega$, defined by $g(y)=[y]_\omega$ for every $y\in Y$, is a heap isomorphism.\end{prop} 

If the equivalent conditions of Proposition~\ref{boh} are satisfied, the heap $X$ is called the semidirect product of $\omega$ and $Y$, denoted $X=\omega\rtimes Y$.

\begin{prop}\label{vhl} Let $(X,[-,-,-])$ be a heap,  $f$ be an idempotent heap endomorphism on $X$. Let $\omega$ be the kernel of $f$, so that $\omega$ is a congruence of the heap $X$, and let $e$ be a fixed element of $Y:=f(X)$. Then $[e]_\omega=f^{-1}(e)$ and there is an action, i.e., a heap homomorphism, $$\alpha\colon (Y,[-,-,-])\to {\rm{Aut}}_{\mathsf{Heap}}([e]_\omega),$$ defined by $\alpha_y(k)=[y,e,[k,y,e]]$ for every $y\in Y$ and $k\in [e]_\omega$. Moreover, $\alpha_e$ is the identity automorphism of $[e]_\omega$.
\end{prop}

The operation $[-,-,-]$ on ${\rm{Aut}}_{\mathsf{Heap}}([e]_\omega)$ is defined by $[f,g,h]=f\circ g^{-1}\circ h$ for every $f,g,h\in {\rm{Aut}}_{\mathsf{Heap}}([e]_\omega)$. 

\smallskip

We are now ready to present the outer semidirect product of heaps:

\begin{prop} Let $K$ and $Y$ be heaps. Let $\alpha\colon Y\to {\rm{Aut}}_{\mathsf{Heap}}(K),$ $\alpha\colon b\in Y\mapsto \alpha_b$, be a heap morphism, and let $y$ be an element of $Y$ that is mapped by $\alpha$ to the identity automorphism $\alpha_y$ of $K$. On the cartesian product $K\times Y$ define a ternary operation $[-,-,-]$ setting $$[(k_1,y_1),(k_2,y_2),(k_3,y_3)]:=([k_1,\alpha_{[y_1,y_2,y]}(k_2), \alpha_{[y_1,y_2,y]}(k_3)], [y_1,y_2,y_3])$$ for every $(k_1,y_1),(k_2,y_2),(k_3,y_3)\in K\times Y$. Then:

{\rm (a)} $K\times Y$ is a heap and $K\times\{y\}$ is a normal subheap of $K\times Y$ isomorphic to~$K$. 

{\rm (b)} For every element $k\in K$, $\{k\}\times Y$ is a subheap of $K\times Y$ isomorphic to $Y$ and the mapping $f\colon K\times Y\to K\times Y$ defined by $f(a,b)=(k,b)$ for every $(a,b)\in K\times Y$ is an idempotent heap endomorphism of $K\times Y$.
\end{prop}

It is clear how to define the (direct) product of two heaps. 

\begin{prop}\label{5.5} Let $X$ be a heap, $\omega$ be a congruence on $X$, $Y$ be a subheap of $X$, and suppose $X=\omega\rtimes Y$. Let $f$ be the corresponding idempotent heap endomorphism of $X$. Fix an element $e\in Y$ and set $K:=[e]_\omega$. 
Then the following conditions are equivalent:

{\rm (a)} The subheap $Y$ of $X$ is normal.
    
{\rm (b)} $X= K\times Y$ (direct product of heaps).

{\rm (c)} 
$[y,e,k]=[k,e,y]$ for every $y\in Y$ and every $k\in K$. 

{\rm (d)} There is an idempotent heap endomorphism $g$ of $X$ whose image is $K$ and such that $g^{-1}(e)=Y$.

{\rm (e)} The heap morphism $\alpha\colon Y\to {\rm{Aut}}_{\mathsf{Heap}}(K)$ of Proposition~{\rm \ref{vhl}} is constantly equal to the identity mapping of $K$.
\end{prop}

In particular, if the equivalent conditions of the previous Proposition~\ref{5.5} hold, then the commutator $[\sim_Y,\omega]$ is the conguence $=$ on $X$. 

\subsection{Left near-trusses}

A {\em left near-truss} $(X,[-,-,-],\cdot{})$ is a set $X$ endowed with a ternary operation $[-,-,-]$ and a binary operation $\cdot$, such that $(X,[-,-,-])$ is a heap, $(X,\cdot{})$ is a semigroup, and {\em  left distributivity} holds, that is, $$x\cdot[y,z,w]=[x\cdot y,x\cdot z,x\cdot w]$$  for every $x,y,z,w\in X$.  Similarly for {\em right near-trusses}, where left distributivity is replaced by {\em right distributivity}: $[y,z,w]\cdot x=[y\cdot x,z\cdot x,w\cdot x]$ for every $x,y,z,w\in X$.  Clearly, the category of  left near-trusses is isomorphic to the category of right near-trusses, it suffices to associate to any left near-truss $(X,[-,-,-],\cdot{})$ its opposite right near-truss $(X,[-,-,-],\cdot{}^{\op})$.

\begin{prop}\label{boh'} Let $X\ne\emptyset$ be a left near-truss, $Y$ be a subnear-truss of $X$, and $\omega$ a congruence on the left near-truss $X$. The following conditions are equivalent:

{\rm (a)} $Y$ is a set of representatives of the equivalence classes of $X$ modulo $\omega$ (i.e., $Y\cap[x]_\omega$ is a singleton for every $x\in X$).

{\rm (b)} There exists an idempotent left near-truss endomorphism of $X$ whose image is $Y$ and whose kernel is $\omega$.

{\rm (c)} For every $a\in X$ and every $c\in Y$ there exist a unique element $b\in X$ and a unique element $d\in Y$ such that $a=[d,c,b]$ and $b\,\omega\, c$.

{\rm (d)} The mapping $g\colon Y\to X/\omega$, defined by $g(y)=[y]_\omega$ for every $y\in Y$, is a left near-truss isomorphism.
\end{prop} 

\section{Inner semidirect product in Universal Algebra}\label{3}

With all the previous examples in mind, it is now clear what must be defined as inner semidirect product in Universal Algebra. Fix a variety $\Cal V$ of type $\Cal F$.

\begin{teo}\label{ciy} Let $A$ be an algebra, $B$ be a subalgebra of $A$, and $\omega$ be a congruence on $A$. The following conditions are equivalent:

{\rm (a)} $B$ is a set of representatives for the elements of $A$ modulo $\omega$, that is, for every $a\in A$ the intersection $[a]_\omega \cap B$ is a singleton.

{\rm (b)} There is an idempotent endomorphism of $A$ whose image is $B$ and whose kernel is $\omega$.

{\rm (c)} There is a surjective homomorphism $ A\to B$ that is the identity on $B$ and whose kernel is $\omega$.

{\rm (d)} The canonical mapping $B\to A/\omega$, $b\mapsto [b]_\omega$, is an isomorphism.\end{teo}

\begin{Proof}
    (a)${}\Rightarrow{}$(b) Suppose that (a) holds. Let $e\colon A\to A$ be the mapping that associates with any $a\in A$ the unique element in  $[a]_\omega \cap B$. It is clear that the image of this mapping is contained in $B$ and, conversely, if $b\in B$, then $b=e(b)$. Hence $B=e(A)$. 
    
    Let us show that the mapping $e$ is a homomorphism. Let $p\colon A^n\to A$ be any $n$-ary operation of $A$ and $a_1,\dots,a_n$ be elements of $A$. In order to show that $e(p(a_1,\dots,a_n))=p(e(a_1),\dots,e(a_n))$, notice that both $e(p(a_1,\dots,a_n))$ and $p(e(a_1),\dots,e(a_n))$ are elements of $B$ because $B=e(A)$ and $B$ is a subalgebra of $A$. Hence, in order to show that they are equal, it suffices to show that $e(p(a_1,\dots,a_n))\,\omega\, p(e(a_1),\dots,e(a_n))$. Clearly $e(a)\,\omega\, a$ for every $a\in A$, so that\linebreak $e(p(a_1,\dots,a_n))\,\omega\, p(a_1,\dots,a_n)\,\omega \,p(e(a_1),\dots,e(a_n))$ because $\omega$ is a congruence of $A$. This concludes the proof that $e\colon A\to A$ is a homomorphism. Finally, $e(a)\,\omega \,a$ implies that $[e(a)]_\omega= [a]_\omega$, so $[e(a)]_\omega\cap B= [a]_\omega\cap B$, whence $e(e(a))=e(a)$. Therefore $e$ is an idempotent endomorphism.

    Finally, it is clear that two elements $a,a'$ of $A$ are mapped to the same element of $B$ if and only if $[a]_\omega= [a']_\omega$, that is, if and only if $a\,\omega\, a'$. This proves that the kernel of $e$ is exactly $\omega$.

    (b)${}\Rightarrow{}$(c) If $e\colon A\to A$ is an idempotent endomorphism of $A$ whose image is $B$ and whose kernel is $\omega$, its corestriction $e|^B\colon A\to B$ is a surjective homomorphism whose kernel is still $\omega$. The mapping $e|^B$ is the identity on $B$, because if $b\in B$, there exists $a\in A$ such that $e(a)=b$. Then $b=e(a)=e(e(a))=e(b)=e|^B(b)$.

    (c)${}\Rightarrow{}$(d) If $g\colon A\to B$ is a surjective homomorphism with kernel $\omega$ that is the identity on $B$, then there is an isomorphism $\overline{g}\colon A/\omega\to B$ such that $\overline{g}([b]_\omega)=b$ for every $b\in B$. Its inverse $\overline{g}^{-1}\colon B\to A/\omega$ is the canonical mapping $B\to A/\omega$, $b\mapsto [b]_\omega$, and is an isomorphism.

    (d)${}\Rightarrow{}$(a) because (a) is just a restatement of (d).
\end{Proof}

If the equivalent conditions in the statement of Theorem~\ref{ciy} hold, we will say that $A$ is the {\em (inner) semidirect-product} of its subalgebra $B$ and its congruence $\omega$, and write $A=B \ltimes\omega$.

\begin{cor}\label{11} Given an algebra $A$ there is a one-to-one correspondence between:

{\rm (a)} The set of all idempotent endomorphisms of $A$.

{\rm (b)} The set of all pairs ($B,\omega)$, where $B$ is a subalgebra of $A$, $\omega$ is a congruence on $A$, and $A=B \ltimes\omega$.\end{cor}

\begin{remark}\label{ru}{\rm Assume $A=B \ltimes\omega$. By Theorem~\ref{ciy}(a), the subalgebra $B$ is a set of representatives for the congruence classes of $A$ modulo~$\omega$. Each congruence class of $A$ modulo $\omega$ is of the form $A_b:=[b]_\omega$ for a unique $b\in B$, and $A$ is the disjoint union of all its subsets $A_b$. Moreover, each $A_b$ has a privileged point $b$ (because $A_b \cap B$ is the singleton $\{b\}$), so that $(A_b,b)$ is a pointed set. For every $n$-ary operation $p\colon A^n\to A$ of $A$ and for every $b_1,\dots,b_n\in B$, one has that $p(b_1,\dots,b_n)\in B$ (this exactly means that $B$ is a subalgebra of $A$) and $p(A_{b_1}\times\dots\times A_{b_n})\subseteq A_{p(b_1,\dots,b_n)} $ (this is equivalent to ``$\omega$ is a conguence for $A$''). Notice that the restriction $A_{b_1}\times\dots\times A_{b_n}\to A_{p(b_1,\dots,b_n)} $ of $p$ is a morphism of pointed sets, that is, a morphism in the category $\Set_*$. Hence the equivalent conditions of Theorem~\ref{ciy} say that there is a partition of $A$ into pointed sets, this partition is compatible with the operations on $A$, and defines a sort of ``grading'' on $A$.}
\end{remark}

In the pairs ($B,\omega)$ relative to semidirect-product decompositions of an algebra $A$, $B$ can be any subalgebra of $A$ that is also a homomorphic image of $A$. Relatively to such a subalgebra $B$, $\omega$ can be any ``complement'' of $B$ in $A$, that is, the kernel of any homomorphism $A\to B$ that is the identity on $B$.

\medskip

By Remark~\ref{ru}, if we have a semidirect-product decomposition $A=B \ltimes\omega$, then {\em the algebra $B$ acts on the corresponding pointed partition of $A$}, in the sense that for every $n$-ary function symbol $f\in\Cal F$ and every $n$-tuple $(b_1,\dots, b_n)\in B^n$, there is a morphism
$f_{(b_1,\dots, b_n)}\colon(A_{b_1},a_{b_1})\times\dots\times(A_{b_n},a_{b_n})\to (A_{f(b_1,\dots, b_n)},a_{f(b_1,\dots, b_n)})$ 
in the category $\Set_*$. It is defined, for every $(x_1,\dots,x_n)\in A_{b_1}\times\dots\times A_{b_n}$, by  $f_{(b_1,\dots, b_n)}(x_1,\dots,x_n)=f(x_1,\dots,x_n)$ (as we have said above, each $f_{(b_1,\dots, b_n)}$ is just the operation $f$, but restricted to $A_{b_1}\times\dots\times A_{b_n}\to A_{f(b_1,\dots, b_n)}$).

\subsection{The special role of constant endomorphisms in semidirect products}

Until now we have stressed the role of idempotent endomorphisms in the study of semidirect-product decompositions of an algebra $A$. Among idempotent endomorphisms of $A$, the easiest but probably most important ones are the {\em constant endomorphisms} of $A$, that is the endomorphisms $f\colon A\to A$ such that $f(a)=f(a')$ for every $a,a'\in A$. Equivalently, the endomorphisms $f\colon A\to A$ for which there exists an element $\overline{a}\in A$ such that $f(a)=\overline{a}$ for every $a\in A$. Clearly, every constant endomorphism is an idempotent endomorphism.

\medskip

We say that an element $a$ of an algebra $A$ is {\em totally idempotent} if $p(a,a,\dots,a)=a$ for every operation $p$ of $A$.

\begin{prop} Let $A$ be an algebra. There are canonical one-to-one correspondences between the following three sets:

{\rm (a)} The set of all constant endomorphisms of $A$.

{\rm (b)} The set of all subalgebras of $A$ of cardinality one.

{\rm (c)} The set of all totally idempotent elements of $A$.
    \end{prop}

    The set $E$ of all idempotent endomorphisms of $A$ can be partially ordered setting, for all $e,f\in E$, $e \le f$ if $e(A)\subseteq f(A)$ and, for all $a,a'\in A$, $ f(a)=f(a')\Rightarrow e(a)=e(a')$. The partially ordered set $(E,\le)$ has the identity automorphism $\id_A$ of $A$ as its greatest element. If $A$ has constant endomorphisms, they are all minimal elements of $(E,\le)$. 

    \begin{teo}\label{qqq} Let $A$ be an algebra, $A=B \ltimes\omega$ be a semidirect-product decomposition of $A$, and $e$ the corresponding idempotent endomorphism of $A$. The following conditions are equivalent for a congruence class $K$ of $A$ modulo $\omega$:

        {\rm (a)} $K$ is a subalgebra of $A$.

        {\rm (b)} The element of $K\cap B$ is totally idempotent.

        {\rm (c)} In the partially ordered set $E$ of all idempotent endomorphisms of $E$, there exists a constant endomorphism $e_0\in E$ such that $e_0\le e$.
    \end{teo}

    \begin{Proof} If (a) holds, then the idempotent endomorphism $e\colon A\to A$, which associates to each congruence class modulo $\omega$ the unique element of the congruence class which also belongs to $B$, restricts to a constant endomorphism $e|^K_K\colon K\to K$, which associates to every element of $K$ the unique element of $K\cap B$. Since images of homomorphisms are subalgebras, it follows that $K\cap B$ is a subalgebra of $B$, so that its unique element is totally idempotent. This proves (b).

    If (b) holds, then the constant mapping $e_0\colon A\to A$ that sends all elements of $A$ to the totally idempotent element $k_0$ is a constant endomorphism of $A$. Moreover $e_0\le e$ because $e_0(A)=\{k_0\}\subseteq B=e(A)$ and, for all $a,a'\in A$, the implication $ e(a)=e(a')\Rightarrow e_0(a)=e_0(a')=k_0$ holds trivially. Thus (c) holds.

    Assume that (c) holds. Let $e_0$ be a constant endomorphism of $A$ with $e_0\le e$. If $k_0$ is the unique element in the image of $e_0$, then $k_0\in e(A)=B$. Moreover $k_0$ is a totally idempotent element because the image $\{k_0\}$ of the homomorphism $e_0$ is a subalgebra of $A$. Now $K$ is a subalgebra of $A$, because if $f$ is an $n$-ary operation of $A$ and $k_1,\dots,k_n\in K$, then $k_1\,\omega\, k_0,\,\dots,\,k_n\,\omega\, k_0$, so that $f(k_1,\dots,k_n)\,\omega \,f(k_0,\dots,k_0)=k_0 $ because $\omega$ is a congruence. Therefore $f(k_0,\dots,k_0)\in K$, and (a) holds.
    \end{Proof}

    \begin{cor} Let $A$ be an algebra and $A=B \ltimes\omega$ be a semidirect-product decomposition of $A$. Then the canonical bijection $B\to A/\omega$ of Theorem~{\rm\ref{ciy}(d)} induces by restriction a bijection between the set of all the elements of $B$ that are totally idempotent and the set of all congruence classes of $A$ modulo $\omega$ that are subalgebras of~$A$.
    \end{cor}

    For instance, in the variety of groups, the unique totally idempotent element of any group is its identity. And of all the cosets $gN$ modulo a normal subgroup $N$, the unique coset that is a subgroup is the subgroup $N$.

    \medskip


   \begin{remark}{\rm The setting described until now in this Section~\ref{3} can be described in an evocative way by the diagram \begin{equation}
    \xymatrix@1{\omega \ar[r] & A \ar[r]^\varphi & {\phantom ,}B, \ar@/^/[l]^\iota}\label{cyul}\end{equation} where $\varphi$ is the surjective homomorphism described in Theorem~\ref{ciy}(c), $\omega$ is the kernel of $\varphi$, $\iota\colon B\hookrightarrow A$ is the inclusion of the subalgebra $B$ into the algebra $A$, $\iota$ is a right inverse of $\varphi$, and $\iota\varphi$ is the corresponding idempotent endomorphism of $A$.

In several cases the congruence $\omega$ on $A$ can be described (can be determined) by one of its congruence classes $K$ that turns out to be a subalgebra of $A$. This is the case of normal subgroups, or of two-sided ideals of a ring (not-necessarily with an identity). On the one hand this allows to replace diagram (\ref{cyul}) with the diagram \begin{equation}
    \xymatrix@1{{K\,} \ar@{^{(}->}[r] & A \ar[r]^\varphi & {\phantom ,}B. \ar@/^/[l]^\iota}\label{7}\end{equation}
On the other hand this also allows us to talk of ``action of the algebra $B$ on the algebra~$K$''. This is possible whenever $K$ is the congruence class modulo $\omega$ of a totally idempotent element.}\end{remark}

\subsection{The example of groups}\label{ss}
Let $\Cal V$ be the varieties of groups (three operations, one binary, one unary, and one $0$-ary). Let $A$ be a group. Every idempotent endomorphism $e$ of $A$ corresponds to a semidirect-product decomposition $A=\omega\rtimes B$, where $B$ is a subgroup of $A$ and $\omega$ a congruence on $A$. In $A$ there is a unique totally idempotent element: the identity $1_A$ of $A$. Since the variety $\Cal V$ of groups is ideal determined, the congruence $\omega$ corresponds to a normal subgroup $N$ of $A$, the congruence class of $1_A$ modulo $\omega$. The corresponding partition of $A$ into pointed subsets is the partition of $A$ in cosets $Na$. More precisely, it is the partition into pointed sets $\Cal{N}=\{\, (Nb,b)\mid b\in B\,\}$. For each of the three $n$-ary operations $p\colon A^n\to A$ on $A$ and each $b_1,\dots,b_n\in B$, we have the restriction $A_{b_1}\times\dots\times A_{b_n}\to A_{p(b_1,\dots,b_n)} $, which is a morphism in the category $\Set_*$. In this case of groups, these restrictions are morphisms of pointed sets $(Nb_1,b_1)\times\dots\times (Nb_n,b_n)\to(Nb_1\dots b_n, b_1\dots b_n)$. But due to the canonical bijections $N\to Nb$, $n\mapsto nb$, these restrictions can be equivalently described as:

(1) Mappings $g_{b_1,b_2}\colon N\times N\to N$. They are defined by the multiplications on the corresponding cosets, so that  $g_{b_1,b_2}(n_1,n_2)=(n_1b_1)(n_2b_2)(b_1b_2)^{-1}=n_1b_1n_2b_1^{-1}$. Notice that, therefore, all these mappings $g_{b_1,b_2}$ do not depend on $b_2$, but only on $b_1$. Moreover, $g_{b_1,1_B}(1_N,-)\colon N\to N$ is the conjugation by $b_1$, written on the left like in Subsection~\ref{2.1} (it is a group homomorphism). Finally notice that $g_{b_1,b_2}(n_1,n_2)=n_1g_{b_1,1_B}(1_N,n_2)$, so that all the mappings $g_{b_1,b_2}$ can be expressed using the only mappings $g_{b_1,1_B}(1_N,-)$ and multiplication in $N$, that is, the mappings $g_{b_1,1_B}(1_N,-)$ are sufficient to parametrize all the mappings $g_{b_1,b_2}$.

(2) The $0$-ary operation corresponds to the element $1_N$ in the pointed set $(N,1_B)$.

(3) The mappings $h_b\colon N\to N$ are defined by the inverse on the corresponding cosets, so that  $h_b(n)=(nb)^{-1}b=b^{-1}n^{-1}b=g_{(b^{-1},1_B)}(1_N,n^{-1})$. Thus all the mappings $h_b$ are also completely determined by the only mappings $g_{b_1,1_B}(1_N,-)$.

Thus, for the special case of groups, all the data $\{g_{(b_1,b_2)}, h_b, 1_N\}$ are completely determined by the group homomorphisms $B\to\Aut_\Gp(N)$, $b\mapsto g_{b,1_B}(1_N,-)$. But notice that the correct notion of ``action'' of an algebra $B$ on an algebra $N$ in the case of one binary operation as is the case of semigroups for instance, is as a collection of morphisms $(A_{b_1},b_1)\times(A_{b_2},b_2)\to (A_{b_1b_2},b_1b_2)$ indexed in $B^2$.

\section{Outer semidirect product}\label{4}

Let $B$ be an algebra in a variety $\Cal V$ of type $\Cal F$. We want to construct an algebra $A$ in $\Cal V$ that is a semidirect product of $B$ and a congruence $\omega$ on $A$. To this end, we take as our initial data the algebra $B$ in the variety and an indexed family $\Cal N=\{\,(A_b,a_b)\mid b \in B\,\}$ of pointed sets, indexed in $B$. Set $$A:=\bigcup_{b\in B} (A_b\times\{b\})\subseteq \left(\bigcup_{b\in B} A_b\right)\times B.$$ For every $n$-ary function symbol $f\in\Cal F$ and every $b_1,\dots, b_n\in B$, let 
$$f_{(b_1,\dots, b_n)}\colon(A_{b_1},a_{b_1})\times\dots\times(A_{b_n},a_{b_n})\to (A_{f(b_1,\dots, b_n)},a_{f(b_1,\dots, b_n)})$$ 
be a morphism in $\Set_*$, that is, a mapping $f_{(b_1,\dots, b_n)}\colon A_{b_1}\times\dots\times A_{b_n}\to A_{f(b_1,\dots, b_n)}$ such that $f_{(b_1,\dots, b_n)}(a_{b_1},\dots,a_{b_n})=a_{f(b_1,\dots, b_n)}$. Then $A$ becomes an algebra of type $\Cal F$ with the  operations defined by $$f((a_{b_1},b_1),\dots, (a_{b_n},b_n))=(f_{(b_1,\dots, b_n)}(a_{b_1},\dots,a_{b_n}), f(b_1,\dots, b_n)).$$ 
The outer semidirect products of $B$ and $\Cal N$ are indexed only by those mappings $f_{(b_1,\dots, b_n)}$ for which the corresponding operations $f$ on $A$ satisfy the identities of $\Cal {V}$. For any such choice of $\Cal N$ and mappings $f_{(b_1,\dots, b_n)}$, we will call the resulting algebra $A$ the {\em (outer) semidirect product} parametrized by $\Cal N$ and the family of mappings $F=\{\,f_{(b_1,\dots, b_{n_f})}\mid f\in\Cal F,\ n_f \ \mbox{\rm the arity of }f,\ b_1,\dots, b_{n_f}\in B\,\}$. It will be denoted by $A=\Cal N\rtimes_F B$. Also, we will denote this assignment of the pointed set $(A_b,a_b)$ to each element $b\in B$ and of the mapping $f_{(b_1,\dots, b_n)}$ to each $n$-ary function symbol $f\in\Cal F$ and each $b_1,\dots, b_n\in B$ in the form $F\colon B\to \Set_*$ because of the similarity of the outer semidirect product we have just defined with a functor into the category $\Set_*$ (see Subsection~\ref{functor} below).

\begin{lem}
Let $A$ be an algebra admitting an inner semidirect-product decomposition $A=\omega\rtimes B$. Then $A$ is isomorphic to the outer semidirect product $\Cal N\rtimes_F B$ where $\Cal N=\{\,([b]_\omega,b)\mid b\in B\,\}$ and $F$ consists of the restrictions of the operations $p$ of $A$ to $[b_1]_\omega\times\dots\times [b_{n_p}]_\omega\to[p(b_1,\dots,b_{n_p})]_\omega$.
\end{lem}

\begin{Proof}
It is straighforward to check that this is a well defined outer semidirect product and that the mapping $a\mapsto(a,b)$, where $b$ is the unique element of $B$ such that $a\in[b]_\omega$, is an isomorphism $A\to\Cal N\rtimes_FB$.
\end{Proof}

An {\em action} of $B$ on $\Cal N$ is therefore an element of 
\begin{equation}\prod_{f\in\Cal F}\left(\prod_{(b_1,\dots,b_n)\in B^n}\hom_{\Set_*}\left((A_{b_1},a_{b_1})\times\dots\times(A_{b_n},a_{b_n}),(A_{f(b_1,\dots, b_n)},a_{f(b_1,\dots, b_n)})\right)\right)\label{cguk}\end{equation} 
for which the corresponding operations $f$ on $A$ satisfy the identities of $\Cal V$. Thus the outer semidirect product is constructed from an algebra $B$, an indexed family $\Cal N=\{\,(A_b,a_b)\mid b \in B\,\}$ of pointed sets indexed in $B$, and a family $F$ of morphisms $f_{(b_1,\dots, b_n)}\colon(A_{b_1},a_{b_1})\times\dots\times(A_{b_n},a_{b_n})\to (A_{f(b_1,\dots, b_n)},a_{f(b_1,\dots, b_n)})$ 
 in $\Set_*$, one for each $n$-ary function symbol $f\in\Cal F$ and each $n$-tuple $(b_1,\dots, b_n)\in B^n$.

\subsection{Outer semidirect products can be extended to functors}\label{functor}

In Category Theory, semidirect product of groups has been generalized to the so-called {\em Grothendieck construction}. The Grothendieck construction can be applied to any  functor from any small category to the category $\Cat$ of small categories. 

As a matter of fact, there is a similarity between the definition of functor and our definition of outer semidirect product of arbitrary algebras. In order to define a functor $F$ from a category $\Cal C$ to a category $\Cal D$, we must:

\noindent(1) associate with each object 
$X$ in $\Cal C$ an object $F(X)$ in $\Cal D$;

\noindent(2) associate with each morphism $f\colon X\to Y$ in $\Cal C$ a morphism  $F(f)\colon F(X)\to F(Y)$ in $\Cal D$;

\noindent(3) and two very natural conditions must hold.

To define an outer semidirect product, we need an algebra $B$ in a variety $\Cal V$ of type $\Cal F$, and then we must:

    \noindent(1) associate with each element $b\in B$ a pointed set $(A_b,a_b)$;
    
    \noindent(2) associate with each operation $f\in\Cal F$, $n_f$-ary operation say, and each $n_f$-tuple $(b_1,\dots,b_{n_f})\in B^{n_f}$, a morphism $f_{(b_1,\dots, b_{n_f})}\colon(A_{b_1},a_{b_1})\times\dots\times(A_{b_{n_f}},a_{b_{n_f}})\to (A_{f(b_1,\dots, b_{n_f})},a_{f(b_1,\dots, b_{n_f})})$
in $\Set_*$,

\noindent in such a way that if we define $A:=\dot\bigcup_{b\in B}A_b=\{\,(x,b)\mid b\in B,\ x\in A_b\,\}$ and the operations $f\colon A^{n_f}\to A$ on $A $ via $$f((a_{b_1},b_1),\dots, (a_{b_n},b_n))=(f_{(b_1,\dots, b_n)}(a_{b_1},\dots,a_{b_n}), f(b_1,\dots, b_n)),$$ then 

\noindent(3) the identities defining the variety $\Cal V$ must hold for the algebra $A$.

\medskip

We all know that the first definition in MacLane's ``Categories for the Working Mathematician'' \cite[p.~7]{ML} is that of metagraph: A {\em metagraph} consists of {\em objects} $a,b,c,\dots,$ {\em arrows} $f,g,h$,\dots, and two operations: (1) {Domain}, which assigns to each arrow $f$ an object $a=\dom f$; (2) {Codomain}, which assigns to each arrow $f$ an object $b=\cod f$. In order to describe here now algebras and their semidirect product, we need the more general notion of metahypergraph: A {\em metahypergraph} consists of {\em objects} $a,b,c,\dots,$ {\em arrows} $f,g,h$,\dots, and two operations: (1) {Domain}, which assigns to each arrow $f$ an $n_f$-tuple $(a_1,\dots,a_{n_f})=\dom f$ of objects, where $n_f$ is a non-negative integer that depends on $f$; (2) {Codomain}, which assigns to each arrow $f$ an object $b=\cod f$. Then an outer semidirect product is a ``functor'' from a metaphypergraph to the category $\Set_*$. More precisely, every algebra $A$ with operations $f_\gamma$ of arity $n_\gamma$ can be viewed as a metaphypergraph with objects the elements of $A$, and one arrow $f_{\gamma, \, (a_1,\dots,a_{n_\gamma})}$ for each operation $f_\gamma$ of $A$ and each $n_\gamma$-tuple $(a_1,\dots,a_{n_\gamma})\in A^{n_\gamma}$. The domain of $f_{\gamma, \, (a_1,\dots,a_{n_\gamma})}$ is the $n_\gamma$-tuple $(a_1,\dots,a_{n_\gamma})$, and its codomain is the element $f_{\gamma}(a_1,\dots,a_{n_\gamma})$ of $A$.

\smallskip

Now the metagraph constructed in the previous paragraph (we will denote it by $A$, the same symbol used for the algebra) can be extended to a category $\Cal C_A$ which we will call the {\em enveloping category} of the metahypergraph $A$. The objects of $\Cal C_A$ are all $n$-tuples $(a_1,\dots,a_n)$ of elements of the algebra $A$ for an arbitrary integer $n\ge0$. A morphism $(a_1,\dots,a_n)\to (b_1,\dots,b_m)$ is any $m$-tuple $(p_1,\dots,p_m)$ of $n$-ary terms \cite[Definitions~10.1 and~10.2]{BS} such that $p_j(a_1,\dots,a_n)=b_j$ for every $j=1,2,\dots,m$. Composition of morphisms is defined by 
$$(q_1,\dots,q_r)\circ(p_1,\dots,p_m)=(q_1(p_1,\dots,p_m), q_2(p_1,\dots,p_m),\dots, q_r(p_1,\dots,p_m)).$$ 

\begin{teo}
Any outer semidirect product $F\colon A\to\Set_*$ can be extended to a functor $G\colon \Cal C_A\to\Set_*$ by defining $$G(a_1,\dots,a_n)=F(a_1)\times\dots\times F(a_n)$$ and $$\begin{array}{l} G(p_1,\dots,p_m)=(F(p_1),\dots, F(p_m))= \\ \qquad =\prod_{j=1}^m\left(F(p_j)\colon F(a_1)\times\dots\times F(a_n)\to F(p_j(a_1,\dots,a_n))\right).\end{array}$$    
\end{teo}

\subsection{The category $\Cal S_B$ of all outer semidirect products of an algebra $B$}

Let $B$ be an algebra in a variety $\Cal V$ of type $\Cal F$. It is possible to construct a category $\Cal S_B$ as follows.

(a) The objects of $\Cal S_B$ are all semidirect products $F\colon B\to \Set_*$, where $$F(b)=(A_b,a_b)$$ for every $b\in B$, 
$$f^F_{(b_1,\dots, b_n)}\colon(A_{b_1},a_{b_1})\times\dots\times(A_{b_n},a_{b_n})\to (A_{f(b_1,\dots, b_n)},a_{f(b_1,\dots, b_n)})$$ for every $n_f$-ary function symbol $f\in\Cal F$ and every $(b_1,\dots, b_n)\in B^n$, and $A=\dot{\bigcup}_{b\in B} (A_b\times\{b\})$ is an algebra in $\Cal V$.

(b) A morphism $\varphi\colon F\to F'$, where $F,F'\colon B\to\Set_*$ are semidirect products of $B$, consists of a morphism $\varphi_b\colon F(b)\to F'(b)$ in $\Set_*$ for every $b\in B$, in such a way that for every $f\in \Cal F$ and every $n_f$-uple $(b_1,\dots, b_n)\in B^n$ the diagram
\begin{equation}
\xymatrix{(A_{b_1},a_{b_1})\times\dots\times(A_{b_n},a_{b_n})\ \ar[r]^{f^F_{(b_1,\dots, b_n)}} \ar[d]_{\varphi_{b_1}\times\dots\times\varphi_{b_n}} & \ (A_{f(b_1,\dots, b_n)},a_{f(b_1,\dots, b_n)}) \ar[d]_{\varphi_{f(b_1,\dots,b_n)}} \\ 
(A'_{b_1},a'_{b_1})\times\dots\times(A'_{b_n},a'_{b_n})\ \ar[r]_{f^{F'}_{(b_1,\dots, b_n)}} & \ (A'_{f(b_1,\dots, b_n)},a'_{f(b_1,\dots, b_n)}) }\label{square}\end{equation} commutes.

\medskip

Also recall the definition of {\em category $\Cal P_B$ of pointed objects over $B$} \cite[p.~28]{[1]}:

(a) Its objects are all triples $(A,\alpha,\beta)$, where $A\in\Cal V$, $\alpha\colon B\to A$, $\beta\colon A\to B$ are homomorphisms in $\Cal V$ and $\beta\alpha=1_B$.

(b) Morphisms $\gamma\colon (A,\alpha,\beta)\to (A',\alpha',\beta')$ are all morphisms $\gamma\colon A\to A'$ in $\Cal V$ for which the diagram $\xymatrix{ & A\ar[rd]^{\beta}\ar[dd]^{\gamma} & \\
B \ar[rd]_{\alpha'}\ar[ru]^{\alpha} & & B \\
& A'\ar[ru]_{\beta'} &}$ commutes
(i.e.,
such that $\gamma\alpha=\alpha'$ and $\beta'\gamma=\beta$).

(c) Composition $\gamma'\circ\gamma$ of morphisms is composition of mappings.

It is straightforward to check that:

\begin{teo}
    The two categories $\Cal S_B$ and $\Cal P_B$ are equivalent. If $(A,\alpha,\beta)$ is an object of $\Cal P_B$, the corresponding object $F\colon B\to\Set_*$ is defined by $F(b)=(\beta^{-1}(b),\alpha(b))$ for every $b\in B$ and each\begin{equation*}\begin{split} f^F_{(b_1,\dots, b_n)}&\colon (\beta^{-1}(b_1),\alpha(b_1))\times\dots\times (\beta^{-1}(b_{n_f}),\alpha(b_{n_f}))\to \\ &\qquad \qquad\to (\beta^{-1}(f(b_1,\dots, b_n)),\alpha(f(b_1,\dots, b_n)))\end{split}\end{equation*} is the restriction of the operation $f^A\colon A^n\to A$ in $A$.
\end{teo}

Notice that the commutativity of all the squares (\ref{square}) exactly means that the corresponding mapping $A=\dot{\bigcup}_{b\in B} F(b) \to A'=\dot{\bigcup}_{b\in B} F'(b) $ is a homomorphism in $\Cal V$.

\subsection{Action of an algebra $B$ on another algebra $K$.} At the beginning of this  Section~\ref{4}, we have defined $A$ as $\bigcup_{b\in B} (A_b\times\{b\})$, so that there is an embedding $A\hookrightarrow \left(\bigcup_{b\in B} A_b\right)\times B.$ A particular case of this construction of semidirect product is that in which the indexed family $\Cal N=\{\,(A_b,a_b)\mid b \in B\,\}$ is such that there exist a set $C$ and an element $c\in C$ such that $A_b=C\times \{b\}$ and $a_b=(c,b)$ for every $b\in B$ (i.e., the semidirect product $F\colon B\to\Set_*$ constantly assigns the pointed set $(C,c)$ with every element $b\in B$). In this case one has a bijection $A\to C\times B$. For example, suppose that $B$ has a totally idempotent element $\overline{b}$, and let $K$ be another algebra in $\Cal V$ with a  totally idempotent element $\overline{k}\in K$. Then any corresponding outer semidirect product consists in determining all structures of algebra in $\Cal V$ over the cartesian product $K\times B$ for which the mapping $K\times B\to K\times B$, $(k,b)\mapsto (\overline{k},b)$, is an idempotent endomorphism, and the two mappings $K\to K\times B$, $k\mapsto (k,\overline{b})$ and $B\to K\times B$, $b\mapsto (\overline{k},b)$ are two algebra embeddings. 

\medskip

The best case is the following. Assume that we have a pointed variety $\Cal V$ 
 (i.e., a variety with exactly one constant) and the subset consisting of only that constant is a subalgebra. Equivalently, in the variety there is a constant $0$ such that, for every $n$-ary operation $f\in\Cal F$, 
$f(0,0,\dots,0)=0$. In such a variety the singletons are the zero objects. For these algebras there is a Galois correspondence between congruence classes of zero (ideals, normal subalgebras) and congruences.

Let $B$ be an algebra in such a variety $\Cal V$ of type $\Cal F$. Take as our initial data the algebra $B$ in the variety and an indexed family $\{\,A_b\mid b\in B\,\}$ of nonempty sets. With the axiom of choice, fix an element $a_b\in A_b$ for each $b\in B$, and construct the disjoint union $A:=\dot{\bigcup}_{b\in B}A_b$. 

Given any two algebras $B$ and $K$ in $\Cal V$, an {\em action} of $B$ on $K$ is any indexed family $m_f$ ($f\in\Cal F$) of mappings $m_f\colon B^{n_f}\to\hom_{\Set_*}(K^{n_f}, K)$ such that the operations
$$f((k_1,b_1),\dots,(k_{n_f},b_{n_f}))=(m_f(b_1,\dots,b_{n_f})(k_1,\dots,k_{n_f}),f(b_1,\dots,b_{n_f}))$$ give an algebra structure on the cartesian product $B\times K$ in such a way that its subalgebras $B\times\{0_K\}$ and $\{0_B\}\times K$ turn out to be canonically isomorphic to the algebras $B$ and $K$, respectively.

This is the case for $\Omega$-groups, which were introduced by Higgins in \cite{Hig}. A variety of $\Omega$-groups is a variety which is
pointed and
has amongst its operations and identities those of the variety of
groups. Examples of $\Omega$-groups are abelian groups, non-unital rings, 
commutative algebras, modules, and Lie algebras. Also see \cite[Example~(2)]{JMT}.

\subsection{When is a semidirect product a direct product.}
Let now $\Cal V$ be a variety of type $\Cal F$, and fix two algebras $B$ and $K$ in $\Cal V$ with two totally idempotent elements $\overline{b}\in B$ and $\overline{k}\in K$. Construct a semidirect product $F_K\colon B\to \Set_*$ with $F_K(b)=(K,\overline{k})$ for every $b\in B$, so that $A:=K\times B$, and $F_K$ associates with any $n$-ary function symbol $f\in\Cal F$ and every $n$-tuple $(b_1,\dots,b_{n})\in B^{n}$, a morphism $f_{(b_1,\dots, b_{n})}\colon K^n\to K$
in $\Set_*$, i.e. a mapping with $f_{(b_1,\dots, b_{n})}(\overline{k},\dots,\overline{k})=\overline{k}$.
On $A=K\times B$ the semidirect-product operations are defined setting $$f((k_1,b_1),\dots,(k_n,b_n))=(f_{(b_1,\dots, b_{n})}(k_1,\dots,k_n),f^B(b_1,\dots, b_{n}).$$ We leave to the reader to verify that:

\begin{prop}
    The semidirect-product algebra $A=K\times B$ is canonically the direct product of $K$ and $B$ if and only if $f_{(b_1,\dots, b_{n})}=f^K$, $f^K$ the $n$-ary operation on $K$, for every $n$-ary function symbol $f\in\Cal F$ and every $(b_1,\dots,b_{n})\in B^{n}$.
\end{prop}

\section{Examples}

\subsection{The example of groups} The elements of (\ref{cguk}) that satisfy the identities of the variety $\Cal V$ parametrize the outer semidirect products of $B$ and $\Cal N$ exactly in the same way as for groups the semidirect products of two groups $B$ and $N$ are parametrized by the group homomorphisms $B\to\Aut_\Gp(N)$.

Let us see this case more in detail. In the case of semidirect product of groups, $B$ is a group (an algebra with one binary operation, one unary operation, and one $0$-ary operation), and the unique totally idempotent element of $B$ is the identity $1_B$ of $B$. Also, the variety of groups is ideal determined, so that the congruence $\omega$ will correspond to a normal subgroup $N$ of $A$, and the corresponding partition of $A$ is the partition into cosets. Thus if, for the indexed family $\Cal N=\{\,(A_b,a_b)\mid b \in B\,\}$ of pointed sets, the sets $A_b$ are not all equipotent, then no family of functions $f_{(b_1,\dots, b_n)}\colon(A_{b_1},a_{b_1})\times\dots\times(A_{b_n},a_{b_n})\to (A_{f(b_1,\dots, b_n)},a_{f(b_1,\dots, b_n)})$ will satisfy the axioms of groups. Hence we can exclude that case immediately and suppose, without loss of generality, that $A_b=N\times \{b\}$ and $a_b=(1_N,b)$ for some fixed group $N$.
Then $A=\bigcup_{b\in B} (A_b\times\{b\})=\bigcup_{b\in B} (N\times\{(b,b)\}), $ so that $A$ is the cartesian product of $N$ and the diagonal $\Delta(B\times B)$ of $B\times B$. The diagonal $\Delta(B\times B)$ is clearly in one-to-one correspondence with $B$ itself, so that we can suppose that $A=N\times B$, the cartesian product of $N$ and $B$.

For the binary function symbol $p\in\Cal F$ denoting the product and every pair $(b_1,b_2)\in B^2$, we must now fix a morphism
$$p_{(b_1,b_2)}\colon(A_{b_1},a_{b_1})\times (A_{b_2},a_{b_2})\to (A_{b_1b_2},a_{b_1b_2})$$ 
 in $\Set_*$. In our case, this means that we have  a mapping $$p_{(b_1,b_2)}\colon (N\times\{b_1\})\times (N\times\{b_2\})\to (N\times\{b_1b_2\})$$ such that $p_{(b_1,b_2)}((1_N,b_1),(1_N,b_2))=(1_N,b_1b_2)$. 
 Due to the canonical bijections $N\to N\times
 \{b\}$, we can equivalently fix mappings $g_{(b_1,b_2)}\colon N\times N\to N$ such that $g_{(b_1,b_2)}(1_N,1_N)=1_N$. Similarly, for the unary operation, we must fix mappings $h_b\colon N\to N$ such that $h_b(1_N)=1_N$. For the $0$-ary operation we take the constant $1_N$ in $N\cong N\times\{1_B\}$. Now we must impose that these data $\{g_{(b_1,b_2)},h_b,1_N\}$ define on the disjoint union $A=N\times B$ a group structure, i.e., that the three axioms of groups are satisfied. As it is easy to see, they correspond to the three conditions

 \noindent(1) $g_{(b_1b_2,b_3)}(g_{(b_1,b_2)}(n_1,n_2),n_3)=g_{(b_1,b_2b_3)}(n_1,g_{(b_2,b_3)}(n_2,n_3))$ for every $b_1,b_2,b_3\in B, \ n_1,n_2,n_3\in N$;
 
 \noindent(2) $g_{(1_B,b)}(1_N,n)=n$ for every $b\in B, \ n\in N$;

 \noindent(3) $g_{(b^{-1},b)}(h(n),n)=1_N$ for every $b\in B, \ n\in N$.
 
 Then the disjoint union $A=N\times B$ becomes a group with respect to the multiplication defined by $$(n_1,b_1)(n_2,b_2)=(g_{(b_1,b_2)}(n_1,n_2), b_1 b_2).$$ 
 
There is a bijection between the set of all group morphisms $B\to\Aut_\Gp(N)$ and the set of all data $\{g_{(b_1,b_2)}, h_b, 1_N\}$. It associates to $\{g_{(b_1,b_2)}, h_b, 1_N\}$ the group morphism $\gamma\colon B\to\Aut_\Gp(N)$, $\gamma\colon  b\mapsto \gamma_b:=g_{(b,1_B)}(1_N,-)$, and to any group morphism $\gamma\colon B\to\Aut_\Gp(N)$ the mappings $g_{(b_1,b_2)}\colon N\times N\to N$ and $h_b\colon N\to N$ defined by $g_{(b_1,b_2)}(n_1,n_2)=n_1\gamma_{b_1}(n_2)$ and $h_b(n)=\gamma_{b^{-1}}(n^{-1})$.

In order to see that this is a bijection, suppose that $\gamma\colon B\to\Aut_\Gp(N)$ is a group homomorphism. Construct the semidirect product $G:=B\ltimes_\gamma N$. Then $G$ is a group, so that we can apply to it what we saw in Example~\ref{ss}, i.e., that the restrictions of the operations on $G$ satisfy the conditions (1), (2), (3) above.

\subsection{The example of digroups} The case of digroups is pretty similar to that of groups. Let $(B,*,\circ)$ and $(N,*,\circ)$ be two digroups (it would be more careful to denote digroups as $(B,*,\circ,{}^{-*},{}^{-\circ},1_B)$, where the operations $*$ and $\circ$ are binary, the operations ${}^{-*}$ and ${}^{-\circ}$ are unary, and $1_B$ is $0$-ary). Similarly to our construction of semidirect product in this paper (but now writing conjugation on the right like in Subsection~\ref{dig}), we construct a digroup structure on the cartesian product $B\times N$, making use of mappings $g^+_{(b_1,b_2)}\colon N\times N\to N$, $g^\circ_{(b_1,b_2)}\colon N\times N\to N$, $h^+_b\colon N\to N$, $h^\circ_b\colon N\to N$, and the constant $(1_B,1_N)\in B\times N$, imposing that they define a digroup structure on $B\times N$.

Notice that on $B\times N$ we now have that $(b_1,n_1)\circ(b_2,n_2)=b_1\circ n_1\circ b_2\circ n_2=b_1\circ b_2\circ b_2^{-1}\circ n_1\circ b_2\circ n_2=(b_1\circ b_2)\circ (\phi_{\circ\,{b_2}}(n_1)\circ n_2)=(b_1\circ b_2,\phi_{\circ\,{b_2}}(n_1)\circ n_2)$, so that \begin{equation}
g^\circ_{(b_1,b_2)}(n_1,n_2)=\phi_{\circ\,{b_2}}(n_1)\circ n_2.\label{(1)}\end{equation} (This is exactly the equality (\ref{circ}).) Again we find that $g^\circ_{(b_1,b_2)}(n_1,n_2)$ does not depend on $b_1$. Moreover, $g^\circ_{(1_B,b)}(n,1_N)=\phi_{\circ\,{b}}(n)$, so that \begin{equation}
    \phi_{\circ\,b}=g^\circ_{(1_B,b)}(-,1_N),\label{quad1}
\end{equation} which yields a group antihomomorphism $(B,\circ)\to\Aut_\Gp(B,\circ)$.

For the operation $+$ on $B\times N$ we have from (\ref{+}) that \begin{equation}g^+_{(b_1,b_2)}(n_1,n_2)=(\Lambda_{b_1* b_2})^{-1}(\phi_{* b_2}(\Lambda_{b_1}(n_1))* \Lambda_{b_2}(n_2)).\label{(2)}
\end{equation}
In particular, $g^+_{(b,1_B)}(1_N,n)=(\Lambda_{b})^{-1}(\Lambda_{b}(1_N)* \Lambda_{1_B}(n))=(\Lambda_{b})^{-1}(n),$ so that $(\Lambda_{b})^{-1}=g^+_{(b,1_B)}(1_N,-)$, or, equivalently, \begin{equation}
    \Lambda_{b}=(g^+_{(b,1_B)}(1_N,-))^{-1}.\label{quad2}
\end{equation}

Finally, $g^+_{(1_B,b)}(n,1_N)=(\Lambda_{b})^{-1}(\phi_{*b}(\Lambda_{1_B}(n))* \Lambda_{b}(1_N))=(\Lambda_{b})^{-1}(\phi_{*b}(n)),$ so that \begin{equation}
   \phi_{*b}(n)=\Lambda_{b}(g^+_{(1_B,b)}(n,1_N)).\label{quad3}
\end{equation} The three formulas (\ref{quad1}), (\ref{quad2}) and (\ref{quad3}) allow determine $\phi_\circ$, $\Lambda$ and $\phi_*$ as functions of the mappings $g^+$ and $g^\circ$. 

Conversely, assume that $\phi_\circ$, $\phi_*$ and  $\Lambda$ are given. From (\ref{(1)}), it is possible to determine the mappings $g^+$, and from (\ref{(2)}) it is possible to determine the mappings $g^\circ$. We leave to the reader to check that $h^\circ_b(n)=\phi_{\circ\,b^{-\circ}}(n^{-\circ})$ and $h^+_b(n)=(\Lambda_{b^{-*}})^{-1}(\phi_{*\,b^{-*}}((\Lambda_b(n))^{-*}))$.

\end{document}